\documentclass[abstracton,a4paper]{scrartcl}

\usepackage[headsepline]{scrpage2}
\usepackage[latin1]{inputenc}
\usepackage{amsfonts}
\usepackage{amsmath}
\usepackage{amssymb}
\usepackage{amsthm}

\usepackage{url}

\usepackage{graphicx}

\usepackage[pdftex,colorlinks,breaklinks,linkcolor=black,citecolor=black,filecolor=black,menucolor=black,urlcolor=black,pdfauthor={Lizhen Ji and Andreas Weber},pdftitle={Lp spectral theory and heat dynamics of locally symmetric spaces}, plainpages=false, pdfpagelabels, bookmarksnumbered=true]{hyperref}

\ihead{L.Ji, A. Weber: $L^p$ spectral theory and heat dynamics of locally symmetric spaces}

\newtheorem{theorem}{Theorem}[section]
\newtheorem{lemma}[theorem]{Lemma}
\newtheorem{proposition}[theorem]{Proposition}
\newtheorem{corollary}[theorem]{Corollary}

\theoremstyle{definition}
\newtheorem{definition}[theorem]{Definition}

\theoremstyle{remark}
\newtheorem{remark}[theorem]{\bf Remark}

\newcommand{\Hy}{\mathbb{H}}
\newcommand{\N}{\mathbb{N}}
\newcommand{\Q}{\mathbb{Q}}
\newcommand{\R}{\mathbb{R}}

\newcommand{\C}{\mathbb{C}}

\newcommand*\e{\mathrm{e}}

\newcommand*\re{\mathrm{Re}}
\newcommand*\im{\mathrm{Im}}

\newcommand*\grad{\mathrm{grad}}
\newcommand*\dive{\mathrm{div}}

\newcommand*\isom{\mathrm{Isom}}

\newcommand*\rank{\mathrm{rank}}

\newcommand*\dom{\mathrm{dom}}
\newcommand*\tr{\mathrm{tr}}
\newcommand*\spa{\mathrm{span}}

\newcommand\ad{\mathrm{ad}}

\newcommand\bP{{\bf P}}

\newcommand*\pddt{\frac{\partial}{\partial t}}

\newcommand*\DMp{\Delta_{M,p}}                 

\newcommand*\DM{\Delta_M}

\title{$\boldsymbol{L^p}$ spectral theory and heat dynamics of locally symmetric spaces}
\author{Lizhen Ji\footnote{Email: lji@umich.edu, 
					Address: 1834 East Hall, Ann Arbor, MI 48109-1043, USA.
		     Partially supported by NSF grant DMS 0604878}\\ 
					{\large Department of Mathematics, University of Michigan}
\and Andreas Weber\footnote{    Email: andreas.weber@math.uni-karlsruhe.de,
						   Address:  Englerstr. 2, 76128 Karlsruhe, Germany.}\\
						   {\large  Institut f\"ur Algebra und Geometrie,
						   Universit\"at Karlsruhe (TH)} }

\date{}

\begin{document}

\maketitle 
\begin{abstract} In this paper we first  derive several results concerning the $L^p$ spectrum
 of locally symmetric spaces with rank one. In particular, we show that there
 is an open subset of $\C$ consisting of eigenvalues of the $L^p$ Laplacian if $p <2$ and that 
 corresponding eigenfunctions are given by certain Eisenstein series. On the other hand, if 
 $p>2$ there is at most a discrete set of real eigenvalues  of the $L^p$ Laplacian. These results
 are used in the second part of this paper in order to show that the dynamics of the $L^p$
 heat semigroups for $p<2$ is very different from the dynamics of the $L^p$ heat semigroups
 if $p\geq 2$.\\
 
\end{abstract}		

\section{Introduction}

The purpose of this paper is twofold. We are first concerned with the $L^p$ spectrum of
the Laplace-Beltrami operator on locally symmetric spaces with rank one and then
we will use the obtained results about the $L^p$ spectrum in order to show that the dynamics of the $L^p$ heat semigroups for $p<2$ is very different from the dynamics of the $L^p$ heat semigroups
if $p\geq 2$.
 
In contrast to the $L^2$ spectrum of the Laplace-Beltrami operator on locally symmetric spaces the 
$L^p$ spectrum (as set), $p\in [1,\infty)$, is only known in special situations. But there are several reasons to study also the $L^p$ spectrum.\\ 
Firstly, from a physical point of view, the natural space to study heat diffusion is $L^1$:
If the function $u(t,\cdot)\geq 0$ denotes the heat distribution at the time $t$, the total amount of heat 
in some region $\Omega$ is given by the $L^1$ norm of $u(t,\cdot)|_{\Omega}$ and hence, the $L^1$ norm has a physical meaning. But, on the other hand, $L^1$ is more difficult to handle than the
reflexive $L^p$ spaces ($p>1$) as the heat semigroup on $L^p$ ($p>1$) is  always bounded analytic
whereas this is in general not true for the heat semigroup on $L^1$.\\ 
Secondly, there are already many results for differential operators on domains of 
euclidean space concerning various aspects of $L^p$ spectral theory, see e.g.
\cite{MR1269649,MR1103113,MR1345724,MR1458713,MR836002,MR869525,MR1602267,MR1673414,MR1824257,MR1224619,MR1420468}.

Another point to mention is the fact that the question whether the $L^p$ spectrum of the Laplace-Beltrami operator on a Riemannian manifold depends non-trivially on $p$ is related to the geometry, in particular the volume growth, of the respective manifold \cite{MR1250269}.
But even if one is only interested in the $L^2$ case, it could be worth studying the $L^p$
spectra, as the knowledge of the $L^p$ spectra yields in some cases further information on
the decay of $L^2$ eigenfunctions  of the Laplace-Beltrami operator  
(see e.g. \cite{MR937635,Ji:2007fk,MR1016445}). Furthermore, 
in \cite{MR1470419} J. Wang was interested in the $L^2$ spectrum of the Laplacian on 
complete Riemannian manifolds with non-negative Ricci curvature. In order to calculate this spectrum he first calculated the $L^1$ spectrum and then used a result of Sturm \cite{MR1250269} to show that the $L^p$ spectrum does not depend on $p$. 

And finally, in the theory of (locally) symmetric spaces there are new phenomena  that cannot   occur in the $L^2$ case. One example for this is
the fact that the $L^p$ spectrum ($p>2$) of  a symmetric space of non-compact type contains an
uncountable number of eigenvalues of the  Laplace-Beltrami operator on $L^p$. More precisely, 
we have  the following result by Taylor:
\begin{theorem}[cf. \cite{MR1016445}]\label{taylor}
	Let $X$ denote a symmetric space of non-compact type. Then  for any 
	$p\in [1,\infty)$ we have 
	$$\sigma(\Delta_{X,p}) = P_{X,p},$$
	where
	$$P_{X,p}=\left\{ ||\rho||^2 - z^2 : z\in\C, |\re z|\leq ||\rho||\cdot |\frac{2}{p}-1|\right\}.$$
	Furthermore, if $p>2$ any point in the interior of the parabolic region $P_{X,p}$ is an 
	eigenvalue for $\Delta_{X,p}$ and eigenfunctions corresponding to these eigenvalues
    are given by spherical functions.
\end{theorem}
For $p=2$ we have $P_{X,2}= [||\rho||^2,\infty)$ and $\Delta_{X,2}$ has no eigenvalues
(for a definition of $\rho$ we refer to Section \ref{symmetric spaces}).
This follows easily by using the Helgason-Fourier transform which turns $\Delta_{X,2}$
into a multiplication operator. 

The $L^p$ spectrum of certain locally symmetric spaces was examined in various articles:
In \cite{MR937635} Davies, Simon, and Taylor determined the $L^p$ spectrum of 
real hyperbolic spaces and certain geometrically finite quotients 
$\Gamma\backslash \Hy^n$. For symmetric spaces of non-compact type with rank one  the results by Lohou\'e and Rychener in \cite{MR689073} should be mentioned. They derived estimates of the
resolvent $(\Delta - z)^{-1}$ of the Laplacian on $L^p$ spaces. More precisely, in the rank one case,
all the complex numbers $z$ are determined such that the respective resolvent is a bounded operator
on $L^p$. 
More general locally symmetric spaces are treated in  \cite{MR1016445,MR2342629} but a precise identification of the $L^p$ spectrum (as set) was only obtained if the universal covering is a rank $1$ symmetric space. Related results are also contained in the classical paper \cite{MR0367561} by Clerc and Stein. \\ 

One of the aims in this paper is to derive similar results as in Theorem \ref{taylor} for  locally 
symmetric spaces $M$ with rank one. For this, we first establish some $L^p$ estimates of
Eisenstein series which are generalized eigenfunctions for the Laplacian and therefore, they can
be regarded as analog to the spherical functions on symmetric spaces of non-compact type. Our
estimates show that plenty of the Eisenstein series are contained in $L^p(M)$ if $1\leq p<2$.
From this result in turn, it will follow that these Eisenstein series are actually (honest) eigenfunctions
for the $L^p$ Laplacian, $1< p<2$.

A point to mention here is that, compared to symmetric spaces, the roles $p<2$ and $p>2$ are 
interchanged. This means that the interior of a similarly defined parabolic region $P_{M,p}$ consists
(besides a discrete set) of eigenvalues if $1<p<2$. If $p\geq 2$ there is only a discrete set of real
eigenvalues possible.

In the second part of this paper we investigate the dynamics of the $L^p$ heat semigroups
$$
 e^{-t\DMp}: L^p(M) \to L^p(M)
$$
on locally symmetric spaces $M$ with rank one and on products of rank one spaces.
As all these spaces $M$ have finite volume, it follows from H\"older's inequality
$L^q(M)\hookrightarrow L^p(M)$ if $p\leq q$ and hence, 
the semigroup $e^{-t\DMp}$ can be regarded as an extension of the semigroup 
$e^{-t\Delta_{M,q}}$ (see Section \ref{heat semigroup} for further details). In this part, we will make use of the results concerning the $L^p$ spectrum in order to show that the $L^p$ heat semigroups on 
locally symmetric spaces with rank one have a certain chaotic behavior if $1<p<2$. This contrasts
the fact that such a behavior is not possible for the $L^p$ heat semigroups if $p\geq 2$.
One reason for this is that the spaces $L^p(M)$ become larger and larger if $p\downarrow 1$
and hence, there is more space for potential chaotic behavior available. In particular, if $p<2$ there
are suddenly plenty of Eisenstein series contained in $L^p(M)$. 

If the ($\Q$-)rank of a locally symmetric space with finite volume is greater than one  the theory of Eisenstein series is more complicated than in the case of rank one but similar results hold in this case, too. 
A first step into this direction is Proposition \ref{products} where  Riemannian products of locally symmetric spaces with rank one are treated. 
The general case is more difficult to work out but as the asymptotic behavior of an automorphic form 
at infinity  is controlled by its constant term a similar result as Proposition \ref{prop main}
can be proven.  We plan to treat the higher $\Q$-rank case in a
future publication.

Analogous results for symmetric spaces of non-compact type have been obtained in \cite{Ji:nr}.

\section{Preliminaries}

\subsection{The heat semigroup on $\boldsymbol{L^p}$ spaces}\label{heat semigroup}

In this section we denote by $M$ an arbitrary complete Riemannian manifold and by
$\Delta=-\dive(\grad)$ the Laplace-Beltrami operator acting on differentiable functions of $M$.
If we denote by $\DM$ the Laplacian with domain $C_c^{\infty}(M)$ (the set of differentiable functions with compact support), this is an essentially self-adjoint operator and hence, 
its closure $\Delta_{M,2}$ is a self-adjoint operator on the Hilbert space $L^2(M)$. 
Since $\Delta_{M,2}$ is positive,
$-\Delta_{M,2}$ generates a bounded analytic semigroup $e^{-t\Delta_{M,2}}$ on $L^2(M)$ 
which can be defined
by the spectral theorem for unbounded self-adjoint operators. The semigroup $e^{-t\Delta_{M,2}}$
is a {\em submarkovian semigroup} (i.e., $e^{-t\Delta_{M,2}}$ is positive and a contraction on 
$L^{\infty}(M)$
for any $t\geq 0$) and we therefore have the following:
\begin{itemize}
\item[(1)] The semigroup $e^{-t\Delta_{M,2}}$ leaves the set $L^1(M)\cap L^{\infty}(M)\subset L^2(M)$ 
		invariant and hence,
		$e^{-t\Delta_{M,2}}|_{L^1\cap L^{\infty}}$ may be extended to a positive contraction semigroup
		$T_p(t)$ on $L^p(M)$ for any $p\in [1,\infty]$.
		 These semigroups are strongly continuous if $p\in [1,\infty)$ and {\em consistent}
		 in the sense that $T_p(t)|_{L^p\cap L^q} = T_q(t)|_{L^p\cap L^q}$. 
\item[(2)] Furthermore, if $p\in (1,\infty)$, the semigroup $T_p(t)$ is a bounded analytic semigroup
		with angle of analyticity $\theta_p \geq \frac{\pi}{2} - \arctan\frac{|p-2|}{2\sqrt{p-1}}$.	
\end{itemize} 
For a proof of (1) we refer to \cite[Theorem 1.4.1]{MR1103113}. For (2) see \cite{MR1224619}.
In general, the semigroup $T_1(t)$ needs not be analytic. However, if $M$ has bounded geometry
$T_1(t)$ is analytic in {\em some} sector (cf. \cite{MR924464,MR1023321}).

In the following, we denote by $-\DMp$ the generator of $T_p(t)$ and by $\sigma(\DMp)$ the spectrum of $\DMp$. Furthermore, we will write
$e^{-t\DMp}$ for the semigroup $T_p(t)$.
Because of (2) from above, 
the $L^p$-spectrum $\sigma(\DMp)$ has to be contained in the sector
\begin{multline*}
\left\{ z\in \C\setminus\{0\} : |\arg(z)| \leq \frac{\pi}{2}-\theta_p\right\}\cup\{0\} \subset\\
     \left\{ z\in \C\setminus\{0\} : |\arg(z)| \leq \arctan\frac{|p-2|}{2\sqrt{p-1}} \right\}\cup\{0\}.
\end{multline*}     

If we identify as usual the dual space of $L^p(M), 1\leq p<\infty$, with 
$L^{p'}(M), \frac{1}{p}+\frac{1}{p'}=1$, the dual operator of $\DMp$ equals $\Delta_{M,p'}$
and therefore we always have $\sigma(\DMp) = \sigma(\Delta_{M,p'})$. 
It should also be mentioned that the family $\DMp, p\geq 1,$ is also consistent, which means that
the restrictions of $\DMp$ and $\Delta_{M,q}$ to the intersection of their domains coincide:
\begin{lemma}\label{lemma consistent}
If $p,q \in [1,\infty)$, the operators $\DMp$ and $\Delta_{M,q}$ are consistent, i.e.
$$ 
\DMp f = \Delta_{M,q} f\qquad\mbox{for any~} f\in \dom(\DMp)\cap\dom(\Delta_{M,q}).
$$
\end{lemma}
\begin{proof}
 Since the semigroups $\e^{-t\DMp}$ and $\e^{-t\Delta_{M,q}}$ are consistent, we have for\linebreak
  $f\in\dom(\DMp) \cap \dom(\Delta_{M,q})$: 
 $$ \frac{1}{t}\left( \e^{-t\DMp}f - f \right) \xrightarrow{||\cdot||_{L^p}}
 	-\DMp f\quad (t \downarrow0)$$
and	
 $$ \frac{1}{t}\left( \e^{-t\DMp}f - f \right) =
 	\frac{1}{t}\left( \e^{-t\Delta_{M,q}}f - f \right)\xrightarrow{||\cdot||_{L^q}} 
 	-\Delta_{M,q} f\quad (t \downarrow0).$$
Furthermore,
$$ \DMp f - \Delta_{M,q} f \,\in\, L^p(M) + L^q(M)$$
and $L^p(M) + L^q(M) = \{h_1 + h_2 : h_1 \in L^p(M), h_2 \in L^q(M)\}$ is a Banach space for the norm
\begin{multline*}
||g||_{L^p+L^q} :=\\ 
	\inf\left\{ ||h_1||_{L^p} + ||h_2||_{L^q} : h_1\in L^p(M), h_2\in L^q(M)
					\mbox{~with~} g= h_1+h_2 \right\}.
\end{multline*}					
 In particular, we obtain
 \begin{multline*}
  || \DMp f - \Delta_{M,q} f ||_{L^p+L^q} \leq \\
    || \frac{1}{t}\left( \e^{-t\DMp}f - f \right) + \DMp f ||_{L^p} +
    || \frac{1}{t}\left( \e^{-t\Delta_{M,q}}f - f \right) + \Delta_{M,q}f ||_{L^q}\\ 
    \longrightarrow 0\quad (t\downarrow 0). 
 \end{multline*}					
\end{proof}
If $M$ is Riemannian manifold with finite volume we have by H\"older's inequality
$L^q(M) \hookrightarrow L^p(M)$ for $1\leq p\leq q\leq \infty$. Hence, by consistency, the 
semigroup $e^{-t\Delta_{M,p}}$ (resp. $\DMp$) can be regarded as extension of the semigroup
$e^{-t\Delta_{M,q}}$ (resp. $\Delta_{M,q}$), $p\leq q$.

We conclude this subsection with a general result that will be needed later. For this we first
recall a uniqueness result of $L^p$ solutions of the heat equation by Strichartz, cf. 
\cite[Theorem 3.9]{MR705991}.
\begin{theorem}\label{thm Strichartz}
 Let $v: (0,\infty)\times M\to \R$ denote a differentiable 
 solution of the heat equation $\pddt u = -\Delta u$ with $v(t,\cdot) \in L^p(M)$ for each $t>0$
 and 
 $||v(t,\cdot)||_{L^p}\leq C e^{Dt}$ for some constants $C$ and $D$, $p\in (1,\infty)$. Then there
 is a function $f\in L^p(M)$ with 
 $$
 v(t,x) = \left(e^{-t\DMp}f\right)(x).
 $$
\end{theorem}
\begin{corollary}\label{corollary Lp eigenfunctions}
 Let $p\in (1,\infty)$ and $f: M\to \R$ denote a differentiable function such that $f\in L^p(M)$ and 
 $\Delta f = \mu f$ for some $\mu\in\R$.  Then $f\in \dom(\DMp)$ and $\DMp f = \mu f$.
\end{corollary}
\begin{proof}
 We put 
 $$
  v(t,x) = e^{-\mu t}f(x).
 $$
 Then the conditions of Theorem \ref{thm Strichartz} are obviously fulfilled and hence,
 $e^{-t\DMp}f (x) =  e^{-\mu t}f(x)$.
 From this it follows 
 $$
 || \frac{1}{t}(e^{-t\DMp}f -f) + \mu f||_{L^p}\to 0\quad (t\to 0)
 $$ 
 and therefore $\DMp f = \mu f$.
\end{proof}
Note, that it is not completely obvious that a differentiable $L^p$ function $f$ that 
satisfies the eigenequation $\Delta f = \mu f$ is  contained in the domain of $\DMp$.  
The purpose of Corollary \ref{corollary Lp eigenfunctions} was to show exactly this.  \\
We do not know whether Theorem \ref{thm Strichartz} and 
Corollary \ref{corollary Lp eigenfunctions} are true in the case $p=1$, too. Therefore, we need to restrict
ourselves in some situations below to the case $p\in (1,\infty)$.    Note also, that if $p=\infty$
a uniqueness result as Theorem \ref{thm Strichartz} cannot hold as there are Riemannian 
manifolds on which non-constant solutions $v(t,x)$ of the heat equation exist such that 
$v(0,x)=1$, cf. \cite[Proposition 7.9]{MR0356254}.

\subsection{Locally symmetric spaces}\label{symmetric spaces}

Let $X$ denote always a symmetric space of non-compact type. Then
$G:= \isom^0(X)$ is a non-compact, semi-simple Lie group with trivial center 
that acts transitively on $X$ and $X=G/K$, where $K\subset G$ is a maximal 
compact subgroup of $G$. We denote
the respective Lie algebras by $\mathfrak{g}$ and $\mathfrak{k}$. Given a corresponding Cartan
involution $\theta: \mathfrak{g}\to\mathfrak{g}$ we obtain the Cartan decomposition
$\mathfrak{g}=\mathfrak{k}\oplus\mathfrak{p}$ of $\mathfrak{g}$ into the eigenspaces of $\theta$. The subspace
$\mathfrak{p}$ of $\mathfrak{g}$ can be identified with the tangent space $T_{eK}X$. We assume,
that the Riemannian metric $\langle\cdot,\cdot\rangle$ of $X$ in $\mathfrak{p}\cong T_{eK}X$ 
coincides with the restriction of the Killing form 
$B(Y,Z) := \tr(\ad Y\circ \ad Z ), Y, Z\in \mathfrak{g},$ to $\mathfrak{p}$. 

For any maximal abelian subspace $\mathfrak{a}\subset \mathfrak{p}$ we refer to 
$\Sigma=\Sigma(\mathfrak{g},\mathfrak{a})$ as the set of restricted roots for the pair $(\mathfrak{g},\mathfrak{a})$,
i.e. $\Sigma$ contains all $\alpha\in \mathfrak{a}^*\setminus\{0\}$ such that
$$ \mathfrak{h}_{\alpha} := \{ Y\in \mathfrak{g} : \ad(H)(Y) = \alpha(H)Y \mbox{~for all~} H\in\mathfrak{a} \}\neq \{0\}.$$
These subspaces $ \mathfrak{h}_{\alpha}\neq \{0\}$ are called root spaces.\\
Once a positive Weyl chamber $\mathfrak{a}^+$ in $\mathfrak{a}$ is chosen, we denote by
$\Sigma^+$ the  subset of positive roots and by 
$\rho:= \frac{1}{2}\sum_{\alpha\in\Sigma^+} (\dim \mathfrak{h}_{\alpha})\alpha$ 
half the sum of the positive roots (counted according to their multiplicity).\\

We say that a locally symmetric space $M = \Gamma\backslash X$, where
$\Gamma \cong \pi_1(M)$ is a non-uniform lattice in $G$ that acts without fixed points on $X$, 
has {\em rank one} if it admits a decomposition
\begin{equation}\label{decomposition}
 M = M_0 \cup Z_1 \cup\ldots\cup Z_k,
\end{equation}
into a compact Riemannian manifold  $M_0$ with boundary and finitely many ends
$Z_i$, $i=1,\ldots k,$ associated to rank one $\Gamma$ cuspidal parabolic
subgroups $P_i\subset G, i = 1,\ldots,k,$ in particular, each $Z_i$ is a fibered cusp
and $M = \Gamma \backslash X$ is a non-compact locally symmetric space with finite volume, cf. also \cite{MR614517}.\\
If $\rank(X) \geq 2$ it follows by Margulis' arithmeticity result that $\Gamma$ is in particular
arithmetic \cite{MR1090825,MR776417}, if $\rank(X)=1$ this needs however not be true \cite{MR932135,MR1090825}.  
We will now recall some basic facts about the geometry and $L^2$ spectral theory of
locally symmetric spaces with rank one in order to fix notation. More details can be found e.g. in
\cite{MR0232893,MR1839581,MR1025165,MR0267041}. For a more general class of
Riemannian manifolds (manifolds with cusps of rank one) which are basically defined via
the decomposition (\ref{decomposition}) above, we refer to W. M\"uller's book
\cite{MR891654}.

\subsubsection{Langlands decomposition and reduction theory}

By
$$
 P = N_PA_PM_P
$$
we denote the Langlands decomposition of  a rank one $\Gamma$ cuspidal parabolic
subgroup $P\subset G$ into a unipotent subgroup $N_P$, a one dimensional 
abelian subgroup $A_P$, and a reductive subgroup $M_P$. 
In the case where $X$ denotes a higher rank symmetric space, this decomposition can be found
in  \cite[III.1.11]{MR2189882}. In the case where $X$ has rank one, we refer to  
\cite[I.1.9]{MR2189882}. See also the remark concerning the comparison of the real 
and rational Langlands decomposition in \cite[III.1.12]{MR2189882}. \\
If we denote by 
$$
 X_P = M_P/ K\cap M_P
$$
the {\em boundary symmetric space}, we have the {\em horocyclic decomposition} of $X$:
$$
  X\cong N_P\times A_P\times X_P.
$$
More precisely, if we denote by $\tau: M_P\to X_P$ the canonical projection, we have an analytic diffeomorphism
\begin{equation}\label{rational horocyclic decomposition}
 \mu: N_P\times A_P\times X_P\to X,\,\, (n,a,\tau(m)) \mapsto nam\cdot x_0,
\end{equation}
for some $x_0 \in X$.
Note, that the boundary symmetric space $X_P$ is a Riemannian product of a symmetric
space of non-compact type by a Euclidean space. 

We denote in the following by $\mathfrak{g}, \mathfrak{a}_P$, and $\mathfrak{n}_P$ the Lie algebras of the (real) Lie groups $G, A_P$, and $N_P$ defined above.
Associated with the pair $(\mathfrak{g}, \mathfrak{a}_P)$ there is  a root system $\Phi(\mathfrak{g}, \mathfrak{a}_P)$.
If we define for $\alpha\in \Phi(\mathfrak{g}, \mathfrak{a}_P)$ the {\em root spaces}
$$ \mathfrak{g}_{\alpha} := \{ Z\in \mathfrak{g} : \ad(H)(Y) = \alpha(H)(Y) \mbox{~for all~} H\in \mathfrak{a}_P \},$$
we have the root space decomposition
$$ 
\mathfrak{g} = \mathfrak{g}_0 \bigoplus_{\alpha\in \Phi(\mathfrak{g}, \mathfrak{a}_P)} \mathfrak{g}_{\alpha}.
$$
Furthermore, the $\Gamma$ cuspidal parabolic subgroup $P$ defines an ordering of 
$\Phi(\mathfrak{g}, \mathfrak{a}_P)$ such that
$$ 
\mathfrak{n}_P = \bigoplus_{\alpha\in \Phi^+(\mathfrak{g}, \mathfrak{a}_P)} \mathfrak{g}_{\alpha}.
$$
The root spaces $\mathfrak{g}_{\alpha}, \mathfrak{g}_{\beta}$ to distinct  positive roots 
$\alpha, \beta\in \Phi^+(\mathfrak{g}, \mathfrak{a}_P)$ are orthogonal with respect to the Killing form:
$$ B(\mathfrak{g}_{\alpha}, \mathfrak{g}_{\beta}) = \{0\}.$$
We also define
$$ 
\rho_P := \sum_{\alpha\in\Phi^{+}(\mathfrak{g}, \mathfrak{a}_P)}(\dim\mathfrak{g}_{\alpha})\alpha.
$$   
If $\rank(X) = 1$, we have $||\rho|| = ||\rho_P||$. If $X$ is a higher rank symmetric space this needs
not always be true. But as we assume that $M = \Gamma\backslash X$ is a rank one
locally symmetric space the root systems $\Phi(\mathfrak{g},\mathfrak{a}_{P_i})$ are canonically isomorphic (cf.  \cite[11.9]{MR0244260}) and moreover, we can conclude 
$||\rho_{P_1}|| = \ldots = ||\rho_{P_k}||$.\\
Furthermore, we denote by $\Phi^{++}(\mathfrak{g}, \mathfrak{a}_P)$ the set of simple positive
roots. Recall, that we call a positive root $\alpha\in \Phi^{+}(\mathfrak{g}, \mathfrak{a}_P)$ simple if
$\frac{1}{2}\alpha$ is not a root.\\
Let us now define for any $t \in \mathfrak{a}_P$ the set
$$
 A_{P,t} = \{ e^{H} \in A_P : \alpha(H) > \alpha(t) \mbox{~for all~} 
 			\alpha\in \Phi^{++}(\mathfrak{g}, \mathfrak{a}_P) \}.
$$ 
Then any end $Z_i$ in the decomposition of the rank one locally symmetric space $M$ 
can be identified with a so-called  {\em Siegel set} 
$$ 
U_i\times A_{P_i,t_i}\times V_i \subset N_{P_i}\times A_{P_i}\times X_{P_i}, 
$$
more precisely, 
$ Z_i = \pi(U_i\times A_{P_i,t_i}\times V_i)$, where $\pi: X\to\Gamma\backslash X$ denotes
the canonical projection and $U_i \subset N_{P_i}, V_i\subset X_{P_i}$ are sufficiently large
bounded subsets.
\subsubsection{$\boldsymbol{L^2}$ spectral theory}\label{L2 spectral theory}

One knows that the $L^2$ spectrum $\sigma(\Delta_{M,2})$ of the Laplace-Beltrami operator 
$\Delta_{M,2}$  on a rank one (non-compact) locally symmetric space $M$ is the union of a point spectrum and an absolutely continuous spectrum.
The point spectrum consists of a (possibly infinite) sequence of eigenvalues
$$0=\lambda_0 < \lambda_1\leq \lambda_2\leq \dots$$
with finite multiplicities such that below any finite number there are only finitely many eigenvalues.
The absolutely continuous spectrum equals $[b^2,\infty)$ where  
$b= ||\rho_{P_1}|| = \ldots = ||\rho_{P_k}|| =: ||\rho_P||$.
In what follows, we denote by $L^2_{dis}(M)$ the subspace spanned by all 
eigenfunctions of $\Delta_{M,2}$
and by $L^2_{con}(M)$ the orthogonal complement of  $L^2_{dis}(M)$ in $L^2(M)$.

The absolutely continuous part of the $L^2$ spectrum is parametrized by generalized eigenfunctions of $\Gamma\backslash X$ which are given by Eisenstein series.
Therefore, we recall several basic facts about Eisenstein series. Our main reference here is
 \cite{MR0232893}.
\begin{definition}
Let $f$ be a measurable, locally integrable function on $\Gamma\backslash X$. 
The {\em constant term}
$f_P$ of $f$ along some rank one $\Gamma$ cuspidal parabolic subgroup $P$ of $G$ is defined as
$$
  f_P(x) = \int_{(\Gamma_P\cap N_P)\backslash N_P} f(nx) dn, 
$$
where the measure $dn$ is normalized such that the  volume of 
$(\Gamma_P\cap N_P)\backslash N_P$ equals one and $\Gamma_P = \Gamma\cap P$.\\
 A function  $f$ on $\Gamma\backslash X$  with the property 
 $f_P=0$ for all rank one $\Gamma$ cuspidal parabolic subgroups $P$ of $G$ is called
 {\em cuspidal} and the subspace of cuspidal functions in $L^2(\Gamma\backslash X)$ is 
 denoted by $L^2_{cus}(\Gamma\backslash X)$.
\end{definition}
It is known that 
$$
L^2_{cus}(M) \subset L^2_{dis}(M)
$$ 
and this inclusion is in general strict as the non-zero constant functions are not contained in
$L^2_{cus}(M)$ if $M$ is non-compact.

Let $P$ be rank one $\Gamma$ cuspidal parabolic subgroup of $G$ and $\Gamma_{M_P}$
the image of $\Gamma_P=\Gamma\cap P$ under the projection
$N_PA_PM_P \to M_P$. Then, $\Gamma_{M_P}$ acts discretely on the
boundary symmetric space $X_P$ and the respective quotient 
$\Gamma_{M_P}\backslash X_P$, called boundary locally symmetric space, has
finite volume. Furthermore, we denote by $\mathfrak{a}_P^*$ the dual of $\mathfrak{a}_P$
and put
$$ 
	\mathfrak{a}_P^{*+} = \{ \lambda \in \mathfrak{a}_P^* :  
		\langle \lambda, \alpha\rangle > 0  \mbox{~for all~} 
		\alpha \in \Phi^{++}(\mathfrak{g}, \mathfrak{a}_P) \}.
$$		
For any $\varphi\in L^2_{cus}(\Gamma_{M_P}\backslash X_P)$ and 
$\Lambda \in \mathfrak{a}_P^*\otimes\C$ with 
$\re(\Lambda) \in \rho_P + \mathfrak{a}_P^{*+}$
we define the {\em Eisenstein series} $E(P|\varphi,\Lambda)$ as follows:
\begin{equation}\label{eisenstein series}
 E(P|\varphi,\Lambda:x) = \sum_{\gamma\in \Gamma_P\backslash \Gamma}
 	e^{(\rho_P+\Lambda)(H_P(\gamma x))}\varphi(z_P(\gamma x)),
\end{equation}
where $\mu(n_P(x),e^{H_P(x)}, z_P(x)) = x$ (cf. (\ref{rational horocyclic decomposition})).
This series converges uniformly for $x$ in compact subsets of $X$ and is holomorphic in 
$\Lambda$. Furthermore,
$E(P|\varphi,\Lambda)$ can meromorphically be continued (as a function of $\Lambda$) to
$\mathfrak{a}_P^*\otimes\C$. 

By definition, the Eisenstein series are $\Gamma$ invariant and hence, they define functions on
$M=\Gamma\backslash X$. We have the following lemma.
\begin{lemma}\label{generalized eigenfunctions}
Let $\varphi\in L^2_{cus}(\Gamma_{M_P}\backslash X_P)$ be an eigenfunction 
of $\Delta_{\Gamma_{M_P}\backslash X_P,2}$ with 
respect to some eigenvalue $\nu$. Then we have
for any $\Lambda\in\mathfrak{a}_P^*\otimes\C$ that is not a pole of 
$E(P|\varphi,\Lambda)$ the following:
$$ 
\Delta E(P|\varphi,\Lambda) = (\nu + ||\rho_P||^2
		-\langle \Lambda,\Lambda\rangle) E(P|\varphi,\Lambda).
$$
\end{lemma}
\begin{proof}
In rational horocyclic coordinates we have for a function $f$ that is constant along $N_P$
the formula
\begin{eqnarray*}
 \Delta f &=& ||\rho_P||^2f  +
 		    e^{\rho_P} \Delta_{A_P}\Big( e^{-\rho_P}f\Big) + \Delta_{X_P}f,
\end{eqnarray*} 
where $\Delta_{A_P}$, and $\Delta_{X_P}$ denote the Laplacians on $A_P$ and 
$X_P$. This follows from an analogous calculation as in the proof of 
\cite[Proposition 3.8]{MR754767} or \cite[Theorem 15.4.1]{MR0231321}.
We therefore obtain
\begin{eqnarray*}
 \Delta e^{(\rho_P+\Lambda)(H_P)}\varphi(z_P) &=&
  		(\nu + ||\rho_P||^2 -\langle \Lambda,\Lambda\rangle) 
		e^{(\rho_P+\Lambda)(H_P)}\varphi(z_P).
\end{eqnarray*}
As $\Delta$ is $G$ invariant, the other terms in the Eisenstein series $E(P|\varphi,\Lambda)$
satisfy this equation too, and hence,  the Eisenstein series $E(P|\varphi,\Lambda)$ itself
satisfies this equation in the region of absolute convergence. By meromorphic continuation, the claim
follows.
\end{proof}

We conclude this section with the description of the constant term of an Eisenstein series
on  a rank one locally symmetric space.\\
The boundary locally symmetric space
$\Gamma_{M_P}\backslash X_P$ is compact for rank one $\Gamma$ cuspidal parabolic subgroups 
$P$  and thus any $L^2$ function on $\Gamma_{M_P}\backslash X_P$ is cuspidal as the
cuspidal condition is empty.

Let $\mu$ be an eigenvalue of some boundary locally symmetric space 
$\Gamma_{M_{P_i}}\backslash X_{P_i}$ and choose an orthonormal basis of the
$\mu$-eigenspace of  $\Gamma_{M_{P_j}}\backslash X_{P_j}$ for any
$j=1,\ldots, k$. The union of these bases is denoted by
$\{\varphi_1^{\mu},\ldots, \varphi_{l(\mu)}^{\mu} \}$. Each $\varphi_{m}^{\mu}$ is associated
to a unique $P_{j(m)}$, i.e. $\varphi_{m}^{\mu}$  is an eigenfunction on 
$\Gamma_{M_{P_{j(m)}}}\backslash X_{P_{j(m)}}$ and defines an Eisenstein series
$E(P_{j(m)}| \varphi_{m}^{\mu},\Lambda)$. The poles of the Eisenstein series in the 
half plane $\re(\Lambda)\geq 0$ are contained in the interval $(0,||\rho_{P}||]$. Furthermore,
if $\Lambda$ is a pole, $||\rho_{P}||^2 - \langle \Lambda,\Lambda\rangle$ is an $L^2$ eigenvalue
of the Laplacian and an eigenfunction is given by the residue of the corresponding 
Eisenstein series \cite{MR0232893,MR1025165}. 

For the constant term 
$E_{P_j}(P_{j(m)}| \varphi_{m}^{\mu},\Lambda)$ of  $E(P_{j(m)}| \varphi_{m}^{\mu},\Lambda)$
along $P_j$ we have (for more details cf. \cite{MR1839581,MR1025165})
\begin{multline}
 E_{P_j}(P_{j(m)}| \varphi_{m}^{\mu},\Lambda)(e^{H_{P_j}}z) = \\
 \delta_{j,j(m)}e^{(\rho_{P_j}+\Lambda)(H_{P_j})}\varphi_{m}^{\mu}(z)  +
  \sum_{i=1}^{l(\mu)}e^{(\rho_{P_j}-\Lambda)(H_{P_j})} 
  		\Big(c_{mi}(\Lambda)\varphi_{i}^{\mu}\Big)(z),
\end{multline}
where $c_{mi}(\Lambda)$ are the entries of the scattering matrix 
$c_{P_j|P_{j(m)}}(w:\Lambda)$.\\
The scattering matrix $c_{P_2|P_1}(w:\Lambda)$  is a bounded (linear) operator
$$
L^2_{cus}(\Gamma_{M_{P_1}}\backslash X_{P_1})\to
L^2_{cus}(\Gamma_{M_{P_2}}\backslash X_{P_2})
$$
and  $[c_{P_2|P_1}(w:\Lambda)\varphi]$ and $\varphi$ 
are eigenfunctions for the same eigenvalue.

\section{$\boldsymbol{L^p}$   spectral theory of locally symmetric spaces}

Let us denote in the following by $M=\Gamma\backslash X$ a locally symmetric space with rank one, by 
$\lambda_0,\ldots, \lambda_r$ the eigenvalues that are strictly smaller than $ ||\rho_P||^2$, and by $P_1,\ldots, P_k$ representatives of rank one $\Gamma$ cuspidal parabolic subgroups of $G$.

For any $p\in [1,\infty)$ we define the parabolic region
\begin{equation}
 P_{M,p} = \left\{ ||\rho_{P}||^2 - z^2 : 
 			z\in\C, |\re z|\leq ||\rho_{P}||\cdot |\frac{2}{p}-1|\right\}\subset \C.
\end{equation} 
Our main concern in this section is to prove (for the notation see the preceding section)
\begin{theorem}\label{main theorem 1} 
Let $M=\Gamma\backslash X$ denote  a locally symmetric space with rank one,
 $p\in (1,2)$, and $\Lambda\in\mathfrak{a}_{P_j}\otimes\C$ with 
 $|\re(\Lambda)(H^0)| < \frac{2-p}{p} ||\rho_{P}||$. 
 Then the Eisenstein series
 $$ E(P_{j(m)}| \varphi_{m}^{\mu},\Lambda)$$
 are eigenfunctions of $\DMp$ with eigenvalue 
 $\mu + ||\rho_{P}||^2 - \left(\Lambda(H^0)\right)^2$ if $\Lambda$ is not a pole.
\end{theorem}
The proof of Theorem \ref{main theorem 1} follows from Proposition \ref{prop main} below,
Lemma \ref{generalized eigenfunctions}, 
and Corollary \ref{corollary Lp eigenfunctions}.\\
Theorem \ref{main theorem 1}  contrasts the following fact.
\begin{proposition}
 Let $M=\Gamma\backslash X$ denote  a locally symmetric space with rank one and $p\geq 2$.
 Then there is at most a discrete set of eigenvalues for $\Delta_{M,p}$. 
\end{proposition}
\begin{proof}
 In the case $p=2$ this is well known and was stated in the previous section. Let now $p>2$ 
 and assume that $\varphi$ is some eigenfunction for $\Delta_{M,p}$. As 
 $L^p(M)\hookrightarrow L^2(M)$ and because of Lemma \ref{lemma consistent} the function
 $\varphi$ is also an eigenfunction for $\Delta_{M,2}$. This shows the claim.
\end{proof}
In what follows, $H^0\in \mathfrak{a}_{P_j}^+$ denotes the unique  element with norm one
(note, that $\dim A_{P_j} =1$), i.e. we have $\rho_{P_j}(H^0)= ||\rho_{P}||$. 
\begin{lemma}
 Let $S_j=U_j\times A_{P_j,t_j}\times V_j$ denote a Siegel set associated with $P_j$, $p\in [1,2)$,   and  let $\Lambda\in\mathfrak{a}_{P_j}\otimes\C$ with 
 $|\re(\Lambda)(H^0)| < \frac{2-p}{p} ||\rho_{P}||$. 
 Then we have 
 $$E_{P_j}(P_{j(m)}| \varphi_{m}^{\mu},\Lambda) \in L^p(S_j)$$
 if $\Lambda$ is not a pole.
\end{lemma}
\begin{proof} 
  The volume form of the symmetric space $X= N_{P_j}\times A_{P_j}\times X_{P_j}$
  with respect to rational horocyclic coordinates is given by 
  $dvol_X = h(z)e^{-2||\rho_{P}||y}dzdy$ where $h>0$ is smooth on 
  $N_{P_j}\times X_{P_j}$ and $\log(a)= yH^0$ for any $a\in A_{P_j}$, cf.
  \cite{MR0338456,MR0387496}.\\
  The integrals
  $$ \int_{t_j}^{\infty} \left|e^{(||\rho_{P}|| \pm \Lambda(H^0))y}\right|^pe^{-2||\rho_{P}||y}dy$$
  are readily seen to be finite if $|\re(\Lambda)(H^0)| < \frac{2-p}{p} ||\rho_{P}||$, and the claim
  follows.
\end{proof}
\begin{lemma}
 The functions 
 $$  E(P_{j(m)}| \varphi_{m}^{\mu},\Lambda)- E_{P_j}(P_{j(m)}| \varphi_{m}^{\mu},\Lambda)$$
 are rapidly decreasing in the Siegel set $S_j$ if $\Lambda$ is not a pole.
\end{lemma}
\begin{proof}
Recall, that we call a $\Gamma$ invariant function $f$ on $X$ rapidly decreasing on a Siegel set
$S$ associated to a rank one $\Gamma$ cuspidal parabolic subgroup $P$ if for all $\Lambda\in\mathfrak{a}^*$ we have 
$\sup_{x\in S} |f(x)|e^{\Lambda(H_{P}(x))} <\infty$, see e.g. \cite[I.2.12]{MR1361168}.
The proof now follows from \cite[p.13]{MR0232893}. In the case $G=SL(2,\R)$ it can also be found
in \cite[7.6]{MR1482800}.
\end{proof}
From the preceding lemmas it follows immediately 
\begin{proposition}\label{prop main}
Let $M=\Gamma\backslash X$ denote  a locally symmetric space with rank one,
 $p\in [1,2)$, and $\Lambda\in\mathfrak{a}_{P_j}\otimes\C$ with 
 $|\re(\Lambda)(H^0)| < \frac{2-p}{p} ||\rho_{P}||$. 
 Then we have
 $$ E(P_{j(m)}| \varphi_{m}^{\mu},\Lambda) \in L^p(M)$$
 if $\Lambda$ is not a pole.
\end{proposition}
From  Theorem \ref{main theorem 1} we immediately obtain
\begin{corollary}\label{corollary spectrum}
There is a discrete set $B\subset P_{M,p}\cap \{z\in \C : \im(z) >0\}$ such that
each point in the interior of $P_{M,p}\setminus B$ is an eigenvalue of $\DMp$ if $p\in (1,2)$
 and $P_{M,p}\subset \sigma(\DMp)$ for all $p\in (1,\infty)$.
\end{corollary}
\begin{proof} Let $\tilde{B}$ denote the (discrete) set of points 
$\Lambda\in \mathfrak{a}_{P_j}\otimes\C$ with 
$- \frac{2-p}{p} ||\rho_{P}||< \re(\Lambda)<0$ such that each $\Lambda\in\tilde{B}$
is a pole for {\em all} Eisenstein series. We define 
$$
 B = \left\{ z\in \C : \exists \Lambda\in\tilde{B} \mbox{~such that~} z = 
 		||\rho_{P}||^2 - \langle \Lambda,\Lambda\rangle \right\}.
$$ 
Then the first statement follows clearly when we choose $\mu =0$. Note, that this is possible 
as the boundary locally symmetric space is compact. Note also that the poles with positive real part
correspond to $L^2$ eigenvalues (see Section \ref{L2 spectral theory}) and hence to $L^p$
eigenvalues if $1< p \leq 2$.
The second statement follows by duality and from the fact that the spectrum is a closed subset of $\C$.
\end{proof}
As  by H\"older's inequality $L^2(M)\hookrightarrow L^p(M)$ for any $p\in (1,2]$, it follows that
each $L^2$ eigenvalue $\lambda_j$ is also an $L^p$ eigenvalue 
(cf. Corollary \ref{corollary Lp eigenfunctions}). Hence, by duality, 
$\{\lambda_j : j\in\N\} \subset \sigma(\Delta_{M,p})$ for all $p\in (1,\infty)$.
We therefore have
$\{\lambda_0,\ldots, \lambda_r\}\cup P_{M,p}\subset \sigma(\DMp)$ 
for all $p\in (1,\infty)$.\\
From Taylor's results  in  \cite{MR1016445} it follows also an ``upper bound'' for the $L^p$ spectrum, i.e. 
$$ 
 \sigma(\DMp) \subset \{\lambda_0,\ldots, \lambda_r\}\cup P_{M,p}',
$$
where 
$$
 P_{M,p}' = \left\{ ||\rho_{P}||^2 - z^2 : 
 			z\in\C, |\re z|\leq ||\rho||\cdot |\frac{2}{p}-1|\right\}\
$$
and hence, $P_{M,p}\subset P_{M,p}'$. Note, that we have
equality here if and only if $||\rho||=||\rho_{P}||$ and this condition
is obviously fulfilled if $X$ is a rank one symmetric space as in this case
$\dim \mathfrak{a}= \dim \mathfrak{a}_{P}$.
But the condition $||\rho||=||\rho_{P}||$ holds also for an important class of
($\Q$-)rank one locally symmetric spaces $M=\Gamma\backslash X$ -- the so-called
Hilbert modular varieties. For these spaces $X$ can be a higher rank symmetric space. 

To introduce this class, let $k$ be a totally real number field,
and ${\cal O}_k$ be the ring of integers of $k$.
Then $SL(2, {\cal O}_k)$ is called a Hilbert modular
group. It is an arithmetic subgroup
of $G=SL(2, \R)\times \cdots \times SL(2, \R) = SL(2,\R)^r$, where
there is one factor for each  embedding of $k$ into $\R$.
The group $SL(2, {\cal O}_k)$ acts properly on
the product $(\Hy^2)^r = \Hy^2\times \cdots \times \Hy^2$ with a finite volume quotient,
which is called a {\em Hilbert modular variety}.
More generally, for any finite index subgroup $\Gamma$ of $SL(2, {\cal O}_k)$,
the quotient of $\Gamma\backslash \Hy^2\times \cdots \times \Hy^2$
is often also called a Hilbert modular variety.

The $\Q$-rank of $M=\Gamma\backslash X$ is equal to $1$ whereas the rank of its associated
symmetric space, i.e., the universal covering of the Hilbert modular
variety, is equal to $r$. Unless the number field $k=\Q$, the
symmetric space has rank strictly greater than $1$.
Let $P_{\infty}$ be the parabolic subgroup of upper triangular
matrices of $SL(2, \R)$.
Then there exists a  minimal {\em rational} parabolic subgroup $\bP$ whose
real locus  $P$ is equal to the product  
$P_{\infty} \times \cdots  \times P_{\infty} = P_{\infty}^r$.
The $\R$-split component $A_P$ in the real Langlands decomposition
of $P$ is equal to the product $A_{\infty} \times \cdots\times A_{\infty}$, where
$$
 A_{\infty}= \left\{ \begin{pmatrix} a &0\\ 0& a^{-1} \end{pmatrix} : a>0 \right\},
$$
and the $\Q$-split component $A_{\bP}$ of the {\em rational} Langlands
decomposition of $\bP$  is  the one dimensional diagonal subgroup
$$
A_{\bP}=
 \left\{(g, \cdots, g) : g\in A_{\infty}\right\} \subset A_{\infty} \times \cdots\times A_{\infty}.
$$
Under the identification of the dual space $\mathfrak a_P$ with itself,
half the sum of positive roots $\rho$ is given by
$\rho=(\frac{1}{2}, \cdots, \frac{1}{2})$.
The orthogonal projection of $\rho$ onto the subspace $\mathfrak{a}_{\bP}$
is equal to  half the sum $\rho_{P}$ of  the {\em rational} roots.
This implies that
$||\rho||=||\rho_{P}||$.
More information about Hilbert modular varieties can be found in  \cite{MR1050763}.\\
From these remarks we obtain
\begin{corollary} 
Let $M$ denote a locally symmetric space with $\Q$-rank one whose universal
covering is a symmetric space of rank one or let $M$ denote a Hilbert modular variety.
Then we have
$$
 \sigma(\DMp) = \{\lambda_0,\ldots, \lambda_r\}\cup P_{M,p}
$$ 
for $p\in (1,\infty)$.
\end{corollary}
%


\section{Heat dynamics}

\subsection{Chaotic semigroups}

There are many different definitions of chaos. We will use the following one which is
basically an adaption of Devaney's definition \cite{MR1046376} to the setting of strongly
continuous semigroups, cf. \cite{MR1468101}.
\begin{definition}
 A strongly continuous semigroup $T(t)$ on a Banach space ${\cal B}$ is called {\em chaotic}
 if the following two conditions hold:
   \begin{itemize}
    \item[\textup{(i)}] There exists an $f\in {\cal B}$ such that its orbit 
    			      $\{T(t)f : t\geq 0 \}$ is dense in ${\cal B}$.	
    \item[\textup{(ii)}] The set of periodic points
    			        $\{ f\in {\cal B} : \exists t>0 \mbox{~such that~} T(t)f=f \}$	 is dense in ${\cal B}$.		   
   \end{itemize}
\end{definition}
\begin{remark}\label{remark1}
 \begin{itemize}
  \item[\textup{(1)}] As with $\{T(t)f : t\geq 0 \}$  also the set $\{T(q)f : q\in\Q_{\geq 0} \}$
  				is dense, ${\cal B}$ is necessarily separable.
  \item[\textup{(2)}] The orbit of any point $T(t)f$ in a dense orbit $\{T(t)f : t\geq 0 \}$ is 
  				again dense in ${\cal B}$. Hence, the set of 
				points with a dense orbit is a dense subset of ${\cal B}$ or empty.
  \item[\textup{(3)}] For a separable Banach space ${\cal B}$ condition (i) in the definition 
  		above is equivalent to {\em topological transitivity}
  				of the semigroup $T(t)$, which means that for any pair of non-empty open subsets
				${\cal U,V} \subset {\cal B}$ there is a $t>0$ with
				$T(t){\cal U}\cap {\cal V}\neq \emptyset$, cf. \cite{MR1468101}.
   \item[\textup{(4)}] If both subsets 
   	$${\cal B}_0 = \left\{f\in {\cal B} : T(t)f\to 0\,\,(t\to \infty) \right\}$$ 
	and
	$${\cal B}_{\infty} = 
	\left\{f\in {\cal B} :  \forall \varepsilon >0\, \exists g\in{\cal B}, t > 0 \mbox{~such that~} 
				||g|| <\varepsilon, ||T(t)g -f||<\varepsilon \right\}$$
			 are dense in ${\cal B}$, the semigroup $T(t)$ has dense orbits.
			 However, this condition is not necessary, cf. \cite{MR1468101}.
   \item[\textup{(5)}] Chaotic semigroups exist only on infinite dimensional Banach spaces. 
   			       When looking at the Jordan canonical form of a bounded operator on
			       a (real or complex) finite dimensional Banach space a proof of this
			       is straightforward, see e.g. \cite[Proposition 11]{MR1685272}.
  \end{itemize}
\end{remark}
A sufficient condition for a strongly continuous semigroup to be chaotic in terms of spectral properties
of its generator was given by Desch, Schappacher, and Webb: 
\begin{theorem}[\cite{MR1468101}] \label{thm dsw}
Let $T(t)$ denote a strongly continuous semigroup on a separable 
 Banach space ${\cal B}$ with generator $A$ and let  $\Omega$ denote an open, connected
 subset of $\C$ with  $\Omega\subset \sigma_{pt}(A)$  (the point spectrum of $A$). 
 Assume that there is  a function $F: \Omega\to {\cal B}$ such that
 \begin{itemize}
  \item[\textup{(i)}] $\Omega \cap i\R \neq \emptyset$.
  \item[\textup{(ii)}] $F(\lambda) \in \ker(A-\lambda)$ for all $\lambda \in \Omega$.
  \item[\textup{(iii)}] For all $\phi \in {\cal B'}$ in the dual space of ${\cal B}$, the mapping
  				$F_{\phi}:\Omega\to \C,\, \lambda\mapsto \phi\circ F $
				is analytic. 
				Furthermore, if for some $\phi \in {\cal B'}$ we have $F_{\phi}=0$
				then already $\phi = 0$ holds.
 \end{itemize}
 Then the semigroup $T(t)$ is chaotic.
\end{theorem}
In  \cite{MR1468101} it was also required that the elements $F(\lambda)$, $\lambda \in \Omega$, are non-zero but as remarked in \cite{MR2128736} this assumption is redundant. \\

In order to make this paper more comprehensive, we include the idea of the proof.
\begin{proof}
 A major role in the proof is played by the following observation: let $U\subset \Omega$
 be any subset that contains an accumulation point. Then it follows  that the
 subset ${\cal B}_{U} = \spa\{ F(\lambda) : \lambda\in U\}$   is dense in ${\cal B}$.
 Indeed, if we suppose the contrary, by the Hahn-Banach Theorem 
 there exists some $\phi\in {\cal B'}, \phi\neq 0,$ such that $\phi\circ F(\lambda)=0$ for all
 $\lambda\in U$. As $U$ contains an accumulation point, it follows from the identity
 theorem for complex analytic functions that $F_{\phi}=0$. But this is a contradiction.\\
 For the subsets $U_0 = \{\lambda\in\Omega : \re(\lambda) < 0\}$,
 $U_{\infty} = \{ \lambda\in\Omega : \re(\lambda) > 0\}$, and
 $U_{per} = \Omega \cap i\Q$ it follows now
 ${\cal B}_{U_0} \subset {\cal B}_0$, ${\cal B}_{U_{\infty}} \subset {\cal B}_{\infty}$, and
 ${\cal B}_{U_{per}} \subset \{ f\in {\cal B} : \exists t>0 \mbox{~such that~} T(t)f=f \}$.
 As all these sets are dense in ${\cal B}$, the proof is complete.
\end{proof}
In the theory of dynamical systems chaotic semigroups are highly unwanted because of their
difficult dynamics. Not much more appreciated are so called subspace chaotic semigroups:
\begin{definition} 
  A strongly continuous semigroup $T(t)$ on a Banach space ${\cal B}$ is called {\em subspace 
  chaotic} if there is a closed, $T(t)$ invariant subspace ${\cal V}\neq \{0\}$ of  ${\cal B}$ such that
  the restriction $T(t)|_{\cal V}$ is a chaotic semigroup on ${\cal V}$.
\end{definition}
Because of Remark \ref{remark1} such a subspace is always infinite dimensional.

Banasiak and Moszy\'nski showed that a subset of the conditions in Theorem \ref{thm dsw} yield a sufficient  condition
for subspace chaos:
\begin{theorem}\textup{(\cite[Criterion 3.3]{MR2128736}).}\label{thm ban}
Let $T(t)$ denote a strongly continuous semigroup on a separable 
 Banach space ${\cal B}$ with generator $A$. Assume, there is an open, connected subset
 $\Omega\subset \C$ and a function $F: \Omega\to {\cal B}, F\neq 0,$ such that
 \begin{itemize}
  \item[\textup{(i)}] $\Omega \cap i\R \neq \emptyset$.
  \item[\textup{(ii)}] $F(\lambda) \in \ker(A-\lambda)$ for all $\lambda \in \Omega$.
  \item[\textup{(iii)}] For all $\phi \in {\cal B'}$, the mapping
  				$F_{\phi}:\Omega\to \C,\, \lambda\mapsto \phi\circ F $
				is analytic. 
 \end{itemize}
 Then the semigroup $T(t)$ is subspace chaotic.\\
 Furthermore, the restriction of $T(t)$ to the $T(t)$ invariant subspace 
 ${\cal V} = \overline{\spa}F(\Omega)$ is chaotic.	  
\end{theorem}
The proof of this result is similar to the proof of Theorem \ref{thm dsw}. Note, that it is not required 
$\Omega\subset \sigma_{pt}(A)$ here, i.e. either $F(\lambda)$ is an eigenvector  or $F(\lambda)=0$ .  But, as explained in \cite{MR2128736}, the assumption $\Omega\subset \C$ is not really weaker.

\subsection{$\boldsymbol{L^p}$ heat dynamics on locally symmetric spaces}

\begin{theorem}\label{thm dynamics}
Let $M=\Gamma\backslash X$ denote a locally symmetric space with $\Q$-rank one.
\begin{itemize}
 \item[\textup{(a)}]
 If $p\in (1,2)$ there is a constant $c_p>0$ such that for any $c>c_p$ the semigroup
 $$
  e^{-t(\DMp -c)}: L^p(M) \to L^p(M)
 $$
 is subspace chaotic.
 \item[\textup{(b)}]
 If $p\geq 2$ and $c\in\R$ the semigroup $e^{-t(\DMp -c)}$ is not subspace chaotic.
\end{itemize}
\end{theorem}
\begin{proof}
For the proof of part (a), we will check the conditions of Theorem  \ref{thm ban}.
If $p<2$, the interior of $P_{M,p}\cap\{ z\in\C : \im(z) <0\}$ consists completely of 
eigenvalues, cf. Corollary \ref{corollary spectrum}, and the apex of $P_{M,p}$
is at the point  
$$
 c_p = ||\rho_{P}||^2 - ||\rho_{P}||^2 \cdot \left(\frac{2}{p} - 1\right)^2 
         = \frac{4||\rho_{P}||^2}{p}\left(1 - \frac{1}{p}\right).
$$
Hence, the point spectrum of $(\DMp - c)$ intersects the imaginary axis for any $c>c_p$.
We assume in the following $c > c_p$ and denote by $\Omega$ the interior of the set
$$
  \left(P_{M,p} - c\right)\cap\{ z\in\C : \im(z) <0\}.
$$
Then, if the usual analytic branch of the square root is chosen, 
$\Omega$ is mapped (analytically) by 
$h(z) = i ||\rho_{P}||^{-1}\sqrt{z+c-||\rho_{P}||^2}$
onto the strip
$$
 \left\{z\in\C : \im(z)>0, 0<\re(z)<(\frac{2}{p}-1)  \right\}.
$$
If we now define
$$
 F: \Omega \to L^p(M),\, z\mapsto E(P_{j(1)}| \varphi_1^0, h(z)\rho_{P_{j(1)}})
$$
the map $F_f: \Omega\to \C, z\mapsto \int_M F(z)(x)f(x) dx$ is analytic as a composition of analytic mappings for all $f\in L^{p'}(M)$. Note, that the integral is always finite as the 
Eisenstein series $F(z)$ are contained in $L^p(M)$. Furthermore, it follows from
Theorem \ref{main theorem 1} that each $F(z)$ is an eigenfunction of $(\DMp-c)$
for the eigenvalue $z$ and the proof of part (a) is complete.\\

If $p\geq 2$, the point spectrum of $\DMp$, and hence of $(\DMp - c)$, is a discrete subset of $\R$.
On the other hand, the intersection of the point spectrum of the generator of a chaotic semigroup 
with the imaginary axis is always infinite, cf.  \cite{MR1855839} and its erratum.
\end{proof}
From the proof of  Theorem \ref{thm dynamics} we immediately obtain
\begin{corollary}\label{corollary chaotic}
Let $M=\Gamma\backslash X$ denote a locally symmetric space with $\Q$-rank one.
If $p\in (1,2)$ and $c>c_p$, the restriction of $e^{-t(\DMp -c)}$ to ${\cal V}$ is
chaotic for any of the subspaces
$$
 {\cal V} = 
 \overline{\spa}\left\{E(P_{j(m)}|\varphi_m^0,h(z)\rho_{P_{j(m)}}\,): z\in\Omega\right\}.
$$
\end{corollary}
\begin{remark}
Let
$$
 {\cal V}_{j(m),\mu} =  
 	\overline{\spa}\left\{E(P_{j(m)}|\varphi_m^{\mu},h(z)\rho_{P_{j(m)}}\,): z\in\Omega\right\}.
$$
As the Eisenstein series for the eigenfunctions $\varphi_m^{\mu}$ lead to $L^p$ eigenvalues in 
the interior of $P_{M,p} + \mu$ it can be shown similarly that the semigroups
$$
 e^{-t(\DMp-c)}\Big|_{{\cal V}_{j(m),\mu}}
$$
are chaotic, if $c > c_p + \mu$.
\end{remark}
\begin{proposition}\label{products}
Let $M_i= \Gamma_i\backslash X_i, i=1,\ldots,k,$ denote locally symmetric spaces with 
$\Q$-rank one and $M=M_1\times\dots\times M_k$ their Riemannian product.
If $p\in (1,2)$ there are a constant $c_p>0$ and a closed $e^{-t\DMp}$-invariant 
subspace ${\cal V}\subset L^p(M)$ such that for all $c>c_p$ the semigroup
$e^{-t(\DMp-c)}\Big|_{\cal V}$ has dense orbits.
\end{proposition}
\begin{proof}
 We restrict ourselves to the case $k=2$. By $L^p(M_1)\otimes L^p(M_2)$ we denote
 the tensor product of the spaces $L^p(M_1)$ and $L^p(M_2)$. For the uniform cross norm
 $g_p$ on this tensor product as in  \cite{Weber:2008ve}   the completion 
  $L^p(M_1)\tilde{\otimes}_{g_p} L^p(M_2)$ of the normed space
 $(L^p(M_1)\otimes L^p(M_2),g_p)$ coincides with $L^p(M_1\times M_2)$. Furthermore,
 we have
 $$
  e^{-t\Delta_{M_1\times M_2,p}} = e^{-t\Delta_{M_1,p}}\otimes e^{-t\Delta_{M_2,p}},
 $$
 cf. \cite{MR0348538,Weber:2008ve}. By Corollary \ref{corollary chaotic} the semigroups
 $$
 T_i(t) = e^{-t(\Delta_{M_i,p}-c_i)}\Big|_{{\cal V}_i}
 $$ 
 are chaotic if $c_i >c_{p,i}$ and the 
 subspaces ${\cal V}_i$ are chosen accordingly ($i=1,2$). Let now 
 $c_p= c_{p,1} + c_{p,2}, c  > c_p$ and choose $c_i,$ with $c_i> c_{p,i}, i=1,2$ and $c=c_1+c_2$. 
 Then it follows from \cite[Corollary 2.2]{Weber:yq} that the tensor product
 $T_1(t) \otimes T_2(t)$ on $L^p(M_1\times M_2) = L^p(M_1)\tilde{\otimes}_{g_p} L^p(M_2)$
 is a strongly continuous semigroup that has dense orbits (it is even recurrent hypercyclic).
 Hence,  the semigroup $e^{-t(\Delta_{M_1\times M_2,p}-c)} = 
  e^{-t(\Delta_{M_1,p}-c_1)}\otimes e^{-t(\Delta_{M_2,p}-c_2)}$
 restricted to the subspace ${\cal V} = {\cal V}_1\otimes {\cal V}_2 \subset L^p(M_1\times M_2)$
 has dense orbits. 
\end{proof}
\paragraph*{Acknowledgements} We want to thank the referee for many valuable comments, in particular, for pointing out that our results hold true also in the case of general rank one locally symmetric spaces with finite volume.
%

\bibliographystyle{amsplain}
\bibliography{dissertation,hypercyclic,symmetricSpaces}

\end{document}